\documentclass[11pt]{amsart}

\pdfoutput=1

\usepackage{amsmath}
\usepackage{fullpage}
\usepackage{xspace}
\usepackage[psamsfonts]{amssymb}
\usepackage[latin1]{inputenc}
\usepackage{graphicx,color}
\usepackage{hyperref}
\usepackage{graphicx}

\usepackage{amsmath}%
\usepackage{amsthm}%
\usepackage{amscd}
\usepackage{amsfonts}%
\usepackage{amssymb}%
\usepackage{graphicx}

\usepackage{mathrsfs}

\usepackage{tikz}
\usetikzlibrary{matrix,arrows}

\usepackage{tikz-cd}

\newtheorem{theorem}{Theorem}[section]

\newtheorem{lemma}[theorem]{Lemma}

\theoremstyle{remark}

\numberwithin{equation}{section}

\newcommand{\Fcal}{\mathscr{F}}

\newcommand{\Lcal}{\mathscr{L}}

\newcommand{\Pro}{\mathbb{P}}

\newcommand{\Z}{\mathbb{Z}}

\newcommand{\F}{\mathbb{F}}
\newcommand{\Q}{\mathbb{Q}}

\newcommand{\pole}{\mathrm{pole}}

\newcommand{\Proj}{\mathrm{Proj}}

\makeatletter
\@namedef{subjclassname@2020}{%
  \textup{2020} Mathematics Subject Classification}
\makeatother

  \DeclareFontFamily{U}{wncy}{}
    \DeclareFontShape{U}{wncy}{m}{n}{<->wncyr10}{}
    \DeclareSymbolFont{mcy}{U}{wncy}{m}{n}
    \DeclareMathSymbol{\Sha}{\mathord}{mcy}{"58}

\begin{document}
\title[]{Some families of polynomials satisfying uniform boundedness for rational preperiodic points}

\author{Hector Pasten}
\address{ Departamento de Matem\'aticas,
Pontificia Universidad Cat\'olica de Chile.
Facultad de Matem\'aticas,
4860 Av.\ Vicu\~na Mackenna,
Macul, RM, Chile}
\email[H. Pasten]{hector.pasten@uc.cl}%

\thanks{Supported by ANID Fondecyt Regular grant 1230507 from Chile.}
\date{\today}

\begin{abstract} We construct explicit non-isotrivial families of polynomials over $\mathbb{Q}$ satisfying uniform boundedness for their rational preperiodic points.
\end{abstract}

\maketitle



\section{Introduction}

In what follows, for simplicity we keep our discussion over $\Q$. 

Given a collection of rational functions $\Fcal\subseteq \Q(z)$ we say that $\Fcal$ has the \emph{UB property} (for uniform boundedness) if there is a bound $B$ depending only on $\Fcal$ such that for every $f\in \Fcal$, the dynamical system determined by $f$ on $\Pro^1(\Q)$ has at most $B$ rational preperiodic points. Morton and Silverman \cite{MortonSilverman} conjectured that if $d\ge 2$ and $\Fcal\subseteq\Q(z)$ is the family of all rational functions of degree $d$, then $\Fcal$ has the UB property. The problem is wide open and even the case of (non-isotrivial) one-parameter families such as
\begin{equation}\label{EqnQuadratic}
\{z^2+c : c\in \Q\}
\end{equation}
is very difficult. See \cite{Poonen} for a detailed study of this quadratic family, and see \cite{DoyleFaber} for an overview of results on the Morton--Silverman conjecture. We remark that it is now known \cite{Looper} that Vojta's conjecture with truncated counting functions implies the UB property for the family \eqref{EqnQuadratic}.

 Nevertheless, non-isotrivial one-parameter families satisfying the UB property have been unconditionally constructed in \cite{Ingram} using finiteness results from Diophantine Geometry. In this note we provide a further construction which mainly uses the Chebotarev density theorem in addition to results in Arithmetic Dynamics.

We write $\Pro^1=\Proj\, \Q[x,y]$ for the projective line over $\Q$. A rational function $\psi\in \Q(t)$ defines an element of the function field of  $\Pro^1$ via the change of variables $t=y/x$ and, therefore, it has a divisor of poles $\pole(\psi)$ defined over $\Q$ on $\Pro^1$ (although the geometric points in the support of $\pole(\psi)$ don't need to be rational).
\begin{theorem}\label{ThmMain} Let $d\ge 2$ and let $\psi\in \Q(t)$ be a rational function satisfying that the support of $\pole(\psi)$ is irreducible on $\Pro^1$ with at least two geometric points. Then the family
$$
\Fcal_{d,\psi} = \{z^d + \psi(c) : c\in \Q\}
$$
has the UB property.
\end{theorem}

We remark the families provided by Theorem \ref{ThmMain} are different from the ones in \cite{Ingram}. Indeed, the constructions in \cite{Ingram} give non-isotrivial one-parameter families of polynomials in $\Q[z]$ where all but finitely many members have no rational preperiodic points other than the point at infinity. However, under the assumptions of Theorem \ref{ThmMain} it can actually happen that infinitely many members of $\Fcal_{d,\psi}$ have a rational preperiodic point other than the point at infinity. This is the case, for instance, when $d=2$ and $\psi(t)=2/(t^2+8)$; the polynomial $z^2+\psi(c)$ has an affine rational fixed point for all $c\in \Q$ satisfying that $c^2+8$ is a square, and there are infinitely many such values of $c$ because the quadric $x^2+8=y^2$ has a rational point (namely, $(x,y)=(1,3)$.)


\section{Proof of the result}

We keep the notation and assumptions of Theorem \ref{ThmMain}.
\begin{lemma}\label{Lemma1} There are infinitely many primes $\ell$ such that for every $c\in \Q$, the arithmetic dynamical system determined by $z^d+\psi(c)$ has good reduction at $\ell$. 
\end{lemma}
\begin{proof} Due to the assumptions on the support of $\pole(\psi)$, we can write 
\begin{equation}\label{EqnGH}
\psi(y/x)=\frac{G(x,y)}{A\cdot H(x,y)^n}
\end{equation}
for some non-zero integer $A$, homogeneous polynomials $G,H\in \Z[x,y]$ with $H$ irreducible over $\Z$ of degree at least $2$, and a positive integer $n$ such that $\deg G = n\deg H$. By the Chebotarev density theorem there is an infinite set $\Lcal$ of primes $\ell$ such that $H\bmod \ell$ has no linear factor in $\F_\ell[x,y]$, see \cite{LenstraStevenhagen} (in fact, for this application it would be enough to use a simpler theorem by Frobenius.) 

Take any $\ell\in \Lcal$ with $\ell\nmid A$. Given any $c\in \Q$, let us write $c=b/a$ with $a,b$ coprime integers and let us replace $x=a$ and $y=b$ in \eqref{EqnGH}. Reducing modulo $\ell$ we see that $\ell\nmid H(a,b)$ and, therefore, $\ell$ does not divide the denominator of $\psi(c)$. 

 Finally, we note that if $\alpha\in \Q$, the arithmetic dynamical system determined by $z^d+\alpha$ has good reduction precisely at the primes not dividing the denominator of $\alpha$. \end{proof}

Let us take $p$ and $q$ two different primes as the ones provided by Lemma \ref{Lemma1} and let $N=(p^2-1)(q^2-1)$. Corollary B in \cite{MortonSilverman} shows that for every $c\in \Q$, the (exact) period of any rational periodic point for $z^d+\psi(c)$ is at most $N$. Theorem \ref{ThmMain} now follows from the next result by Doyle and Poonen (cf. Theorem 1.10 in \cite{DoylePoonen}):

\begin{theorem} Let $d\ge 2$ and $N\ge1$ be positive integers. There is a bound $B$ depending only on $d$ and $N$ such that the following holds: Given any $\alpha\in \Q$ such that $z^d+\alpha$ has all its rational periodic points of exact period at most $N$, we have that $z^d+\alpha$ has at most $B$ rational preperiodic points.
\end{theorem}
 

\section{Acknowledgments}

Supported by ANID Fondecyt Regular grant 1230507 from Chile. We thank Joseph Silverman and Patrick Ingram for comments on an earlier version of this manuscript.


\end{document}